\documentclass[10pt,a4paper,reqno]{amsart}
\usepackage{relsize,exscale}
\usepackage{hyperref}
\newcommand*{\affaddr}[1]{#1}
\newcommand*{\affmark}[1][*]{\textsuperscript{#1}}
\usepackage{url}
\usepackage{graphics,layout,indentfirst,amsmath,amsthm,amssymb}
\newtheorem{theorem}{Theorem}%[section]

\theoremstyle{theorem}

\newtheorem{lemma}[theorem]{Lemma}

\numberwithin{equation}{section}
\makeatletter
\@namedef{subjclassname@2010}{2010 Mathematics Subject
Classification} \makeatother
\begin{document}
\title[On Random Fourier-Hermite Transform associated with stochastic process ]
{ On Random Fourier-Hermite Transform associated with stochastic process }

%R\lowercase{eal} Z\lowercase{eros} \lowercase{of} A\lowercase{lgebraic} P\lowercase{olynomials\; with} N\lowercase{onidentical}\\ D\lowercase{ependent} R\lowercase{andom} C\lowercase{oefficients}}
\author[B. Mangaraj, S. Sahoo ]
{B\lowercase{haratee} M\lowercase{angaraj}\affmark[1] \lowercase{and}
S\lowercase{abita} S\lowercase{ahoo}\affmark[2]\\
\affaddr{\affmark[1] D\lowercase{epartment} \lowercase{of} M\lowercase{athematics}, S\lowercase{ambalpur} U\lowercase{niversity}, O\lowercase{disha}, I\lowercase{ndia}}\\
\affaddr{\affmark[2] D\lowercase{epartment} \lowercase{of} M\lowercase{athematics}, S\lowercase{ambalpur} U\lowercase{niversity}, O\lowercase{disha}, I\lowercase{ndia}.
}\\
{\affmark[1] \lowercase{mangarajbharatee@suniv.ac.in} }\\
{\affmark[2] \lowercase{sabitamath@suniv.ac.in} } \\
 }
\date{}

 \subjclass[2010]{Primary: 42A38;  Secondary:
40G15}
\begin{abstract}
Liu and Liu in 2007 introduced the Fourier - Hermite transform $\sum a_{n}\lambda_{n}^{R}\psi_{n}(t)$ which is a random Fourier - Hermite series with random variables $\lambda_{n}^{R}$ choosen randomly from the unit circle of $\mathbb{C}$, where $\psi_{n}(t)$ are Hermite functions and $a_{n}$ are Fourier - Hermite coefficients of an $L^{2}(\mathbb{R})$ function. They used it in image encryption and decryption and expected its application in general signal and image processing. This motivated us to investigate more on random Fourier - Hermite transform by replacing the random variables $\lambda_{n}^{R}$ by some other random variables. It leads to address two problems. First to focus on convergence of random Fourier - Hermite series. Secondly to investigate on finding Fourier transform of the sum function of these random Fourier - Hermite series. The random variables those has been choosen are Fourier - Hermite coefficients of stochastic process. They are independent if associated with Wiener process and dependent if associated with symmetric stable process. The scalars $a_{n}$ are Fourier - Hermite coefficients of functions of suitable $L^{p}$ spaces. The Fourier transform of the sum functions are found out which is possible in case of $p = 2$ only.
\end{abstract}
\keywords{Symmetric stable process, Stochastic integral, Random Fourier-Hermite series, Convergence in mean, Convergence in probability.}

\maketitle{}
\section{Introduction}
\setcounter{equation}{0}
Let $\Omega$ be a sample space with elements $\omega$, $\mathcal{B}$ be a Borel field of sets $E$ contained in $\Omega$, $P(E)$ be the probability measure on $\mathcal{B}$. Let $W(\delta, \omega)$ for $(\delta, \omega) \in \mathbb{R}\times\Omega$ be the Wiener process. Hunt\cite{3} in 1951 could define stochastic integral $\int_{-\infty}^{\infty}f(\delta)W(d\delta, \omega)$ for $f \in L^{2}(\mathbb{R})$ as a limit of Riemann - Stieltjes sum which is a random variable. This led him to define random Fourier transform
\begin{eqnarray*}
\int_{-\infty}^{\infty}e^{i\delta t}f(\delta)dW(\delta, \omega)
\end{eqnarray*}
 for $f \in L^{2}(\mathbb{R})$, associated with Wiener process.
 In 2007, Liu and Liu \cite{9} introduced  random Fourier transform in Hermite polynomials. He used the Fourier - Hermite series expansion of $f \in L^{2}(\mathbb{R})$ and introduced the random Fourier transform as a series $\sum a_{n}\lambda_{n}^{R}\psi_{n}(t)$ in Hermite functions $\psi_{n}(t)$, where $\lambda_{n}^{R}$ are choosen randomly from the unit circle in $\mathbb{C}$, $a_{n}$ are the Fourier - Hermite coefficients of $f \in L^{2}(\mathbb{R})$. They used it in image encryption and decryption. They expects that this concept of RFT will be significant not only for digital and optical image encryption but also in general signal and image processing. Nayak, Pattanayak and Mishra\cite{13}, and Pattanayak and Sahoo\cite{15} have established the convergence of random series $\sum_{n = -\infty}^{\infty}a_{n}r_{n}(\omega)e^{int}$ in the sense of probability and mean respectively, to the stochastic integral $\int_{0}^{2\pi}f(t-s)dX(t, \omega)$ if $X(t, \omega)$ is a symmetric stable process of index $\gamma$, $1 \leq \gamma \leq 2$, $r_{n}(\omega)$ are Fourier coefficients of $X(t, \omega)$ and $a_{n}$ are Fourier coefficients of $f \in L^{p}[0, 2\pi]$, $p \geq \gamma$. The sum function of these series are now random functions. The wide spread application of classical orthogonal polynomials and importance of RFS in signal processing, optics etc motivated us to find the Fourier transform of random functions associated with orthogonal Hermite polynomials.

 The Hermite polynomials $H_{n}(t)$ of degree $n$ are defined as $H_{n}(t) = (-1)^{n}e^{t^{2}}(\frac{d}{dt})^{n}\{e^{-t^{2}}\}$. The Hermite polynomials are mutually orthogonal with respect to the weighted inner product with the weight function $e^{-t^{2}}$, i.e.
\begin{eqnarray}\label{1.4a}
\int_{-\infty}^{\infty}H_{m}(t)H_{n}(t)e^{-t^{2}}dt = \sqrt{\pi}2^{\frac{n}{2}}n!\delta_{mn}
\end{eqnarray}
where $\delta_{mn}$ is the Kronecher's delta function.
The normalized Hermite functions, sometimes called the Hermite - Gaussian functions of degree $n \in \mathbb{N}_{0}$ \cite{5a, 5b, 5c, 13a} or simply Hermite functions are defined as
\begin{eqnarray}\label{2a}
\psi_{n}(t) := \frac{1}{\sqrt{2^{n}n!\sqrt{\pi}}}H_{n}(t)e^{-\frac{t^{2}}{2}}, n \geq 0, t \in \mathbb{R},
\end{eqnarray}
so that
\begin{eqnarray}\label{2b}
\int_{-\infty}^{\infty}\psi_{m}(t)\psi_{n}(t)dt = \delta_{mn}.
\end{eqnarray}
The Hermite functions $\psi_{n}(t)$ form a complete orthonormal basis on the Hilbert space of square integrable functions\cite{12} and $L^{1} \bigcap L^{2}$\cite{1f} on the real line $\mathbb{R}$.
Thus to each $f \in L^{2}(\mathbb{R})$ we have an associated Fourier - Hermite expansion
\begin{equation}\label{2aad}
f(t) := \sum_{k =0}^{\infty}a_{k}\psi_{k}(t),
\end{equation}
where
\begin{eqnarray}\label{1.8a}
a_{k} := \int_{-\infty}^{\infty}f(t)\psi_{k}(t)dt
\end{eqnarray}
such that $\|s_{n} - f\|_{2} \rightarrow 0$ as $n \rightarrow \infty$. Here $s_{n}$ is the $n^{th}$ partial sum of the series (\ref{2aad}), and $a_{n}$ are the Fourier - Hermite coefficients of $f$.
The normalized Hermite functions (\ref{2a}) are eigen functions of the Fourier transform(FT) $\mathcal{F}$ defined by
\begin{equation}
\mathcal{F}(f(t)) := \int_{-\infty}^{\infty}f(t)e^{- i2\pi s t}dt, ~~s~~\in~~\mathbb{R}
\end{equation}
that is
\begin{equation}\label{1.3}
\mathcal{F}(\psi_{n}(t)) := \lambda_{n}\psi_{n}(t),~~ t~~ \in ~~\mathbb{R}
\end{equation}
or more generally
\begin{eqnarray}\label{1.8aa}
\mathcal{F}(\psi_{n}(ax+b)) := \frac{\lambda_{n}}{|a|}e^{-\frac{by}{a}}\psi_{n}(\frac{y}{a}),~a~\neq~ 0~,~b~\in \mathbb{R},~~y~~\in~~\mathbb{R}.
\end{eqnarray}
and $n \in \mathbb{N}_{0}$, $\mathbb{N}_{0} = {0,1,2....}$. Here the eigen values $\lambda_{n}$ are $e^{-\frac{in\pi}{2}}$.
Because the Hermite-Gaussian functions form a complete orthonormal set and a function $f \in L^{2}(\mathbb{R})$ has the Fourier - Hermite expansion (\ref{2aad}), we can express the Fourier transform of $f \in L^{2}(\mathbb{R})$ in terms of these eigen functions as follows:
\begin{eqnarray}\label{1.9a}
\mathcal{F}(f(t)) = \sum_{n = 0}^{\infty}a_{n}\lambda_{n}\psi_{n}(t),
\end{eqnarray}
which exists \cite{1,9}, and $\lambda_{n} = e^{\frac{-in\pi}{2}}$ takes only $4$ possible values $\{1,-1, i, -i\}$ . The $\beta$th order fractional Fourier transform(FrFT) $\mathcal{F}^{\beta}$, $\beta \in \mathbb{R}$ has the same eigen functions as that of the Fourier transform $\mathcal{F}$ but its eigen values are the $\beta$th power of the eigen values of the Fourier transform operator $\mathcal{F}$ \cite{1}. Hence
\begin{eqnarray}\label{1.10a}
\mathcal{F}^{\beta}(\psi_{n}(t)) := \lambda_{n}^{\beta}\psi_{n}(t),
\end{eqnarray}
and more generally
\begin{eqnarray}\label{1.11a}
\mathcal{F}^{\beta}(\psi_{n}(ax+b)) := \frac{\lambda_{n}^{\beta}}{|a|}e^{-\frac{by}{a}}\psi_{n}(\frac{y}{a}),
\end{eqnarray}
$a \neq 0$, $b \in \mathbb{R}$, $y \in \mathbb{R}$, $n \in \mathbb{N}_{0}$, $\mathbb{N}_{0} = 0, 1, 2, ...$.

Bultheel and Martunez \cite{1} could show that, the FrFT of rational order $\beta$  of an arbitrary function $f$ in $L^{2}(\mathbb{R})$ can also be expressed as,
\begin{equation}\label{1aaaa}
\mathcal{F}^{\beta}[f(t)] := \sum_{n = 0}^{\infty} a_{n}\lambda_{n}^{\beta}\psi_{n}(t).
\end{equation}
It can be written in integral form as
\begin{equation}\label{2aaa}
\mathcal{F}^{\beta}[f(t)] := \int_{-\infty}^{\infty}f(s)K_{\beta}(t,s) ds,
\end{equation}
with kernel
\begin{equation}\label{2aaaa}
K_{\beta}(t,s) = \sum_{n = 0}^{\infty}\psi_{n}(s)\psi_{n}(t)\lambda_{n}^{\beta} = \sqrt{1 - i \cot \frac{\beta \pi}{2}}\exp\Bigg[ i\pi \Big(  \frac{s^{2} + t^{2}}{\tan \frac{\beta \pi}{2}} - \frac{2st}{\sin \frac{\beta \pi}{2}}\Big)\Bigg].
\end{equation}
Liu and Liu \cite{9} proposed to extend the eigen values $\lambda_{n}^{\beta}$ of rational order $\beta$ in the FrFT (\ref{1aaaa}) to irrational order and thus giving rise to the series
\begin{equation}\label{A}
\sum_{n = 0}^{\infty} a_{n}\lambda_{n}^{\mathcal{R}}\psi_{n}(t),
\end{equation}
where the infinite number of eigen values $\lambda_{n}^{\mathcal{R}}$ are now randomly choosen values on the unit circle in $\mathbb{C}$. He introduced the series (\ref{A}) as the random Fourier transform(RFT) of $f \in L^{2}(\mathbb{R})$, which is a random Fourier series in Hermite polynomials. What would happen if the random variables $\lambda_{n}^{R}$ in (\ref{A}) be replaced by some other random variables ?

This set before us two things to do. First to investigate the convergence of the random Hermite series
 $\sum_{n = 0}^{\infty}a_{n}r_{n}(\omega)\psi_{n}(t)$, where $a_{n}$ are scalars, $\psi_{n}(t)$ are Hermite functions and $r_{n}(\omega)$  are random variables, and  to find their sum function. Secondly to investigate on existence of the Fourier transform of their sum functions.

In this context the literatures on Fourier - Hermite expansion of functions and literatures on random Fourier series give a way to explore this problem.

Pollard \cite{16} in 1948 showed that, if $f \in L^{2}(\mathbb{R})$ and $\|f\|_{2} = \big(\int_{-\infty}^{\infty}|f(x)|^{2}e^{-x^{2}}dx \big)^{\frac{1}{2}}$, then
\begin{equation}\label{1.8}
\parallel s_{n} - f\parallel_{2} \rightarrow 0
\end{equation}
as $n \rightarrow \infty$ where $s_{n}$ is the nth partial sum of the series $\sum_{k = 1}^{\infty}a_{k}\psi_{k}(t)$.

In 1965, Askey and Wainger \cite{1a} extended this result to a larger class of functions. They established that for $f$ in $L^{p}(\mathbb{R})$, $\frac{4}{3} < p < 4$,
\begin{equation}\label{1.8}
\parallel f - \sum_{k = 0}^{n}a_{k}\psi_{k}(t)\parallel_{p}\rightarrow 0
\end{equation}
as $n \rightarrow \infty$ where $\|f\|_{p}$ is the usual norm of $L^{p}$ that is $\|f\|_{p} = \{\int_{-\infty}^{\infty}|f|^{p}dx\}^{\frac{1}{p}}$.

Pawlak and Stadtm\"{u}ller \cite{17} in 2008 extended this result further to $L^{p}(\mathbb{R})$, $1< p < \infty$ and obtained that the series (\ref{2aad}) converges in $L^{p}(\mathbb{R})$ for $t \in \mathbb{R}$ a.e..

With regards to random Fourier series, it is well known that $\int_{a}^{b}f(s)dW(s, \omega)$ is defined in quadratic mean if $f \in L^{2}[a, b]$ and is a random variable, where $W(s, \omega)$ is a Wiener process. In particular if $f(s) = e^{-ins}$ then
\begin{eqnarray}\label{1.18}
A_{n}(\omega) := \int_{a}^{b}e^{-ins}dW(s, \omega),~~n~~ \in ~~\mathbb{Z}~~,
\end{eqnarray}
exists and are independent random variables.
  Further the stochastic integral
  \begin{equation}\label{B}
  \int_a^bf(s)dX(s,\omega)
  \end{equation}
  is defined in the sense of probability and is a random variable, if $X(s,\omega)$, $s \in \mathbb{R}$ is a continuous stochastic process with independent increments and $f$ is a continuous function in $[a,b]$\cite{11}. If $X(s,\omega),$ is a symmetric stable process with independent increment of index $\gamma \in (1,2],$ then the integral (\ref{B}) is defined in the sense of convergence in mean if $f \in L^{p}[a, b]$ for all $p \geq \gamma$ \cite{6}. In this case  for $f(s) = e^{-ins}$ in $[0, 2\pi]$, the integrals
  \begin{equation}\label{C}
  A_n(\omega) = \frac{1}{2\pi}\int_0^{2\pi}e^{-ins} dX(s,\omega),~n~\in~\mathbb{Z}.
  \end{equation}
  exist and are random variables, which are no longer remains independent. $A_{n}(\omega)$ are called Fourier coefficient of $X(s, \omega)$.
  There is an extensive study on convergence of the random series
  \begin{equation}\label{D}
  \sum_{n = -\infty}^ \infty a_n A_n(\omega)e^{int}
  \end{equation}
  where the random coefficients $A_{n}$ are associated with some stochastic process and $a_{n}$ are the Fourier coefficients of some function \cite{13, 15}. The mode of convergence of the random series (\ref{D}) depends on the stochastic process $X(t, \omega)$ with which $A_{n}(\omega)$ are associated, and on the function $f$ whose Fourier coefficients are $a_{n}$.

  In reference to the work of Liu and Liu \cite{8, 9}, Dash and Pattanayak \cite{2} and Nayak, Pattanayak and Mishra\cite{13}, an attempt has been made in this article to explore more on random Fourier - Hermite transform.

Section 2 in this article deals with the Wiener process. Convergence in quadratic mean of the random Fourier series
  \begin{eqnarray}\label{1c}
  \sum_{n = 0}^{\infty} \tilde{b}_{n}\tilde{B}_{n}(\omega)\tilde{\psi}_{n}^{\alpha}(s)
  \end{eqnarray}
  in transformed Hermite functions is established where $\tilde{\psi}_{n}^{\alpha}$ are the normalized transformed Hermite functions which are orthogonal with respect to the weight function $w(s)$,  $\tilde{B}_{n}(\omega)$ are random variables choosen to be the transformed Hermite coefficients of the Wiener process defined as $\tilde{B}_{n}(\omega) := \int_{-\infty}^{\infty}\tilde{\psi}_{n}^{\alpha}(s)\sqrt{w(s)}dW(s, \omega)$ which are found to be independent and the scalars $\tilde{b}_{n}$ are the transformed Hermite coefficients of a function $f$ in the weighted space $L^{2}_{w}(0, 1)$. The Fourier transform of generalized Hermite functions and transformed Hermite functions are discussed in the first two subsections. These helps to find the Fourier transform of the sum function of the series (\ref{1c}).

  Section 3 is associated with the symmetric stable process of index $\gamma$, $1 < \gamma \leq 2$. The functions in the class $L^{1}(\mathbb{R}) \bigcap L^{\gamma}(\mathbb{R})$, $\frac{4}{3}< \gamma \leq 2$ are taken into consideration. The existence of the integral $\int_{-\infty}^{\infty}f(t)dX(t, \omega)$ in the sense of mean is established for $L^{1}(\mathbb{R}) \bigcap L^{\gamma}(\mathbb{R})$. The orthogonal functions are considered to be the normalized Gaussian functions $\psi_{n}(t)$. The random variables $A_{n}(\omega)$ are choosen to be the Fourier - Hermite coefficients of the symmetric stable process $X(t, \omega)$ defined as $A_{n}(\omega) = \int_{-\infty}^{\infty}\psi_{n}(t)dX(t, \omega)$ which are found to be not independent and the scalars $a_{n}$ are the Fourier - Hermite coefficients of a function $f$ in $L^{1}(\mathbb{R}) \bigcap L^{\gamma}(\mathbb{R})$, $\frac{4}{3} < \gamma \leq 2$. The convergence of the series $\sum_{n = 0}^{\infty} a_{n}A_{n}(\omega)\psi_{n}(t)$ is established. The Fourier transform of the sum function of the series is then found out for the case $\gamma = 2$ only.

\section{\textbf{Random Fourier - Hermite transform associated with Wiener process}}
\setcounter{equation}{0}
\subsection{Fourier transform of generalized Hermite function:}

Since Hermite functions are defined on the infinite interval, we consider transformed Hermite functions which are defined on the interval $(0, 1)$.
Using the auxiliary parameter $\alpha > 0$, the generalized Hermite function is defined by Bao and Shen \cite{1b}, Gou and Shen \cite{2b} and Xian and Wang \cite{19} as
 \begin{eqnarray}\label{2ab}
\psi^{\alpha}_{n}(t) = \frac{\sqrt{\alpha}}{\sqrt{2^{n}n!\sqrt{\pi}}}H_{n}(\alpha t)e^{-\frac{\alpha^{2} t^{2}}{2}}, n \geq 0, t \in \mathbb{R}
\end{eqnarray}
which are normalized so that
\begin{eqnarray}\label{2ba}
\int_{-\infty}^{\infty}\psi^{\alpha}_{n}(t)\psi^{\alpha}_{m}( t)dt = \delta_{nm}.
\end{eqnarray}
 Since $\psi_{n}^{\alpha}(t)$ are complete in $L^{2}(\mathbb{R})$, for any function $g \in L^{2}(\mathbb{R})$, we have an associated expansion
\begin{equation}\label{2aac}
g(t) := \sum_{n =0}^{\infty}b_{n}\psi^{\alpha}_{n}(t),
\end{equation}
where
 \begin{equation}\label{2}
b_{n} := \int_{-\infty}^{\infty}g(t)\psi^{\alpha}_{n}(t)dt, n \in \mathbb{N}_{0}.
\end{equation}
The function $\psi^{\alpha}_{n}(t)$ is the $n^{th}$ eigenfunction of the following Sturm - Liouville problem,
 \begin{eqnarray}
 -f^{\prime\prime}(t) + \alpha^{4}t^{2}f(t) - \alpha^{2}(2k+1)f(t) = 0.
 \end{eqnarray}
By using (\ref{1.8aa}), it can be shown that
\begin{eqnarray}
\mathcal{F}(\psi^{\alpha}_{n}(t)) := \frac{1}{\sqrt{\alpha}}\lambda_{n}\psi_{n}(\frac{s}{\alpha}),~\alpha~>~ 0~,
\end{eqnarray}
as $\psi^{\alpha}_{n}(t) := \sqrt{\alpha}\psi_{n}(\alpha t)$.
\\The $\beta^{th}$ ordered Fourier transform of $\psi^{\alpha}_{n}(t)$ can be written as
\begin{eqnarray}
\mathcal{F}^{\beta}(\psi^{\alpha}_{n}(t)) := \sqrt{\alpha}\frac{\lambda^{\beta}_{n}}{|\alpha|}\psi_{n}(\frac{s}{\alpha}),~\alpha~>~ 0~,~0~<\beta~<~1.
\end{eqnarray}
Therefore for any function $g \in L^{2}(\mathbb{R})$, we can write the Fourier-Hermite transform and $\beta^{th}$ Fourier - Hermite transform of (\ref{2aac}) as
\begin{eqnarray}
\mathcal{F}(g(t)) := \sum_{n = 0}^{\infty}\frac{1}{\sqrt{\alpha}}b_{n}\lambda_{n}\psi_{n}(\frac{t}{\alpha}).
\end{eqnarray}
and
\begin{eqnarray}
\mathcal{F}^{\beta}(g(t)) := \sum_{n = 0}^{\infty}\frac{1}{\sqrt{\alpha}}b_{n}\lambda^{\beta}_{n}\psi_{n}(\frac{t}{\alpha})
\end{eqnarray}
respectively.
\\The last expression can be represented in the integral form as
\begin{eqnarray}
\mathcal{F}^{\beta}[g(t)] := \int_{-\infty}^{\infty}\frac{1}{\alpha}g(t)\psi_{n}^{\alpha}(t)\psi_{n}^{\alpha}(\frac{s}{\alpha})\lambda_{n}^{\beta}dt = \int_{-\infty}^{\infty}\frac{1}{\alpha}g(t)K_{\beta}(t, s)dt
\end{eqnarray}
with the kernel
\begin{eqnarray}
K_{\beta}(t, s) := \sum_{n = 0}^{\infty}\psi_{n}^{\alpha}(t)\psi_{n}^{\alpha}(\frac{s}{\alpha})\lambda_{n}^{\beta}= \sqrt{1 - i\cot \frac{\beta \pi}{2}} \exp \Bigg[i\pi \Big( \frac{(\alpha t)^{2} + (\frac{s}{\alpha})^{2}}{\tan \frac{\beta \pi}{2}} - \frac{2ts}{\sin \frac{\beta \pi}{2}}\Big)\Bigg].
\end{eqnarray}
\subsection{Fourier transform of Transformed Hermite function :}

Since conformal maps are functions transforming one domain to another, A. Saadatmandi, Z. Akbari \cite{17a}, considered the one to one conformal map
 \begin{eqnarray}
 \tilde{\omega} = \phi(z) = \ln(\frac{z}{1-z})
 \end{eqnarray}
 and its inverse
 \begin{eqnarray*}
 z = \phi^{-1}(\tilde{\omega}) = \frac{e^{\tilde{\omega}}}{1+ e^{\tilde{\omega}}} .
 \end{eqnarray*}
 $\phi(z)$ maps the interval $(0, 1)$ on the real line to the total real line $\mathbb{R}$. The range of $\phi^{-1}$ on the real lines is
 \begin{eqnarray*}
 I = \{ \phi^{-1}(t) : -\infty < t < +\infty\} = (0,1).
 \end{eqnarray*}
 The basis functions on $I$ are taken to be the transformed Hermite functions,$\tilde{\psi}^{\alpha}_{n}(x)$ defined as
 \begin{eqnarray}\label{2abe}
\tilde{\psi}^{\alpha}_{n}(x) \equiv \psi^{\alpha}_{n} \circ \phi(x) = \psi^{\alpha}_{n}(\phi(x)).
\end{eqnarray}
They defined the space
\begin{eqnarray*}
L_{\tilde{\omega}}^{2}(I) = \{ g : I \rightarrow \mathbb{R} |~~ \textnormal{g ~~is~~ measurable~~ and} ~~\|g\|_{L_{\tilde{w}}^{2}(I)} < \infty\},
\end{eqnarray*}
where $\tilde{w}(x)$ is a non-negative, integrable, real valued function over the interval $I$ and
\begin{eqnarray}
\|g\|_{L_{\tilde{w}}^{2}(I)} := \big(\int_{0}^{1}|g(x)|^{2}\tilde{w}(x)dx\big)^{\frac{1}{2}} = \big( \int_{0}^{1}|g(x)\sqrt{\tilde{w}(x)}|^{2}dx\big)^{\frac{1}{2}}
\end{eqnarray}
is the norm induced by the inner product of the space ${L_{\tilde{w}}^{2}(I)}$. The function $\tilde
{\psi}_{n}^{\alpha}$, $n \in \mathbb{N}_{0}$ are the $n$th eigen function of the singular Sturm - Liouville problem
\begin{eqnarray*}
-\{x(1-x)v^{'}(x)\}^{'} + \frac{\alpha^{4}\phi^{2}(x)}{x(1-x)}v(x) - (2n + 1)\frac{\alpha^{2}}{x(1-x)}v(x) = 0.
\end{eqnarray*}
Consider the weight function to be $w(x) = \phi^{'}(x) = \frac{1}{x(1 - x)}$.
The system ${\tilde{\psi}_{n}^{\alpha}(x)}_{n = 0}^{\infty}$ are mutually orthogonal over the interval $(0, 1)$ with respect to this weight function $w(x)$ \cite{17a}. Since the system ${\tilde{\psi}_{n}^{\alpha}(x)}_{n = 0}^{\infty}$ is complete in ${L_{w}^{2}(I)}$, so any function $\tilde{g} \in {L_{w}^{2}(0, 1)}$, can have the series expansion in transformed Hermite functions, that is,
\begin{eqnarray}\label{2.12a}
\tilde{g}(x) := \sum_{n = 0}^{\infty}\tilde{b}_{n}\tilde{\psi}_{n}^{\alpha}(x),
\end{eqnarray}
where
\begin{eqnarray}
\tilde{b}_{n}(x) := \int_{ 0}^{1}\tilde{g}(x)\tilde{\psi}_{n}^{\alpha}(x)w(x)dx,
\end{eqnarray}
for $n \in \mathbb{N}_{0}$.

The Fourier transform of $\tilde{\psi}_{n}^{\alpha}(x)$ can be computed to be $\frac{1}{\sqrt{\alpha}}\lambda_{n}\psi_{n}(\frac{\phi(x)}{\alpha})$. Therefore the Fourier transform of $\tilde{g}(x)$ can be calculated as
\begin{eqnarray}
\mathcal{F}(\tilde{g}(x)) = \sum_{n = 0}^{\infty}\tilde{b}_{n}\frac{1}{\sqrt{\alpha}}\lambda_{n}\psi_{n}(\frac{\phi(x)}{\alpha}).
\end{eqnarray}
Similarly the fractional Fourier transform of transformed Hermite function can be computed to be $\frac{1}{\sqrt{\alpha}}\lambda_{n}^{\beta}\psi_{n}(\frac{\phi(x)}{\alpha})$, and hence the fractional Fourier transform of $\tilde{g}(x)$ can be calculated as $\sum_{n = 0}^{\infty}\tilde{b}_{n}\frac{1}{\sqrt{\alpha}}\lambda_{n}^{\beta}\psi_{n}(\frac{\phi(x)}{\alpha}$ .
\subsection{Random Fourier - Hermite transform :}
Let $W(s, \omega), s >0$ be the Wiener process. Since
\begin{eqnarray*}
E \big| \int_{0}^{1} f(s)dW(s, \omega)\big|^{2} = c^{2}\int_{0}^{1} | f(s)|^{2} ds
\end{eqnarray*}
for $f \in L^{2}(0, 1)$, where $c$ is a constant associated with the normal law of increment of the process $W(s, \omega)$, for $s \in (0, 1)$, the stochastic integral $\int_{0}^{1}f(s)\sqrt{w(s)}dW(s, \omega)$ is defined in quadratic mean for $f \in L^{2}_{w}(0, 1)$. In particular if $f(s) = \tilde{\psi}_{n}^{\alpha}(s) \in L_{w}^{2}(0, 1)$ then the stochastic integrals
\begin{eqnarray}\label{2.15}
\tilde{B}_{n}(\omega) = \int_{0}^{1}\tilde{\psi}_{n}^{\alpha}(s)\sqrt{w(s)}dW(s, \omega),~~ n ~~\in~~ \mathbb{N}_{0}
\end{eqnarray}
 exist and are random variables. $\tilde{B}_{n}(\omega)$ are called the Fourier - Hermite coefficients of $W(s, \omega)$. It is shown in the following theorem that they are independent.
\begin{theorem}\label{t1}
If $W(s, \omega)$, $s > 0$ is a Wiener process and $ \tilde{\psi}_{n}^{\alpha}(s)$ is the transformed Hermite function defined as in (\ref{2abe}), then the random variables $\tilde{B}_{n}(\omega)$, $n \in \mathbb{N}_{0}$ are independent.
\end{theorem}
\begin{proof}
The Wiener process $W(s, \omega)$ has orthogonal increments and if $f, g \in L^{2}(0, 1)$, then by Doob \cite{2a},
\begin{eqnarray*}
E \big(\int_{0}^{1}f(s)dW(s, \omega)\overline{\int_{0}^{1}g(s)dW(s, \omega)}\big) = \int_{0}^{1}f(s)\overline{g(s)}ds
\end{eqnarray*}
where $\overline{g(s)}$ is the complex conjugate of $g(s)$.
Thus,
\begin{eqnarray*}
E(\tilde{B}_{n}(\omega)\overline{\tilde{B}_{m}(\omega)}) &=& E \big(\int_{0}^{1}\tilde{\psi_{n}^{\alpha}}(s)\sqrt{w(s)}dW(s, \omega)\overline{\int_{0}^{1}\tilde{\psi}_{m}^{\alpha}(s)\sqrt{w(s)}dW(s, \omega)} \big)
\\&=& \int_{0}^{1}\tilde{\psi}_{n}^{\alpha}(s)\overline{\tilde{\psi}_{m}^{\alpha}(s)}w(s)ds.
\end{eqnarray*}
The orthogonality of $\tilde{\psi}_{n}^{\alpha}$ proves that $\tilde{B}_{n}(\omega)$ are independent random variables.
\end{proof}
The following theorem establishes the convergence of the random series $\sum \tilde{b_{n}}\tilde{B}_{n}(\omega)\tilde{\psi}_{n}^{\alpha}(t)$.
\begin{theorem}\label{T2}
Let $W(t, \omega)$, $t > 0$, $\omega \in \Omega$, be the Wiener process and  $\tilde{g} \in L^{2}_{w}(0,1)$. Let
\begin{eqnarray}\label{2.14}
\tilde{b}_{n} = \int_{0}^{1}\tilde{\psi}_{n}^{\alpha}\tilde{g}(t)w(t)dt,
\end{eqnarray}
be the Fourier - Hermite coefficient of $\tilde{g}$ and the  random variables $\tilde{B}_{n}(\omega)$ be the Fourier - Hermite coefficient of $W(t, \omega)$, defined as in (\ref{2.15}),
then the random  series
\begin{eqnarray}\label{2.15a}
\sum_{n = 0}^{\infty}\tilde{b}_{n}\tilde{B}_{n}(\omega)\tilde{\psi}_{n}^{\alpha}(t)
\end{eqnarray}
in transformed Hermite functions, converges in quadratic mean to the stochastic integral $\int_{0}^{1}\tilde{g}(s, t)dW(t, \omega)$.
\end{theorem}
\begin{proof}
Let
$$\tilde{S}_n(t, \omega):=
\sum_{k=0}^n \tilde{b}_{k}\tilde{B}_{k}(\omega)\tilde{\psi}_{k}^{\alpha}(t)$$
be the $n^{th}$ partial sum of the series $\sum_{k = 0}^{\infty}\tilde{b}_{k}\tilde{B}_{k}(\omega)\tilde{\psi}_{k}^{\alpha}(t)$. Substitution of the integral form of $\tilde{B}_{n}(\omega)$ in $\tilde{S}_n(t, \omega)$, we get
 \begin{eqnarray*}
\tilde{S}_{n}(t, \omega) &:=& \sum_{k = 0}^{\infty} \tilde{b_{k}}\big(\int_{0}^{1}\tilde{\psi}_{k}^{\alpha}(s)\sqrt{w(s)}\big)\tilde{\psi}_{k}^{\alpha}(t)w(s)dW(s, \omega)
\\&=& \int_{0}^{1}\sum_{k = 0}^{\infty}\tilde{b_{k}}\tilde{\psi}_{k}^{\alpha}(s)\tilde{\psi}_{k}^{\alpha}(t)dW(s, \omega).
\end{eqnarray*}
Denote the partial sum of the transformed Hermite series $\sum_{n = 0}^{\infty}\tilde{b}_{n}\tilde{\psi}_{n}^{\alpha}(t)$ of $\tilde{g} \in L_{w}^{2}(0, 1)$ as
$$\tilde{t}_n(t) := \sum_{k=0}^n \tilde{b}_{k}\tilde{\psi}_{k}^{\alpha}(t).$$\\Since $\tilde{\psi}_{n}^{\alpha}(t)$ are bounded in $(0, 1)$, $\sum_{k=0}^{\infty} \tilde{b}_k \tilde{\psi}_k^{\alpha}(s)\tilde{\psi}_k^{\alpha}(t)$ exists. Denote it as $\tilde{g}(s, t)$. Let us denote the partial sum $$\sum_{k=0}^{n} \tilde{b}_k \tilde{\psi}_k^{\alpha}(s)\tilde{\psi}_k^{\alpha}(t)$$ as $\tilde{t}_{n}(s, t)$.\\Now
\begin{eqnarray*}
\tilde{S}_{n}(t,\omega) &=&
\sum_{k=0}^n \tilde{b}_k
\tilde{B}_k(\omega)\tilde{\psi}_k^{\alpha}(t)
\\
&=& \int_{0}^{1}\tilde{t}_n(s,t)\sqrt{w(s)}dW(s,\omega).
\end{eqnarray*}
\\We know that $\int_{0}^{1}\tilde{g}(s, t)\sqrt{w(s)}dW(s, \omega)$ exists in the sense of convergence in quadratic mean.
\\Now,
\begin{eqnarray*}
&&E \Bigg( \Bigg|\int_{0}^{1} \tilde{g}(s,t)\sqrt{w(s)}dW(s,\omega)-\tilde{S}_n(t,\omega) \Bigg|^{2} \Bigg)
\\&=& E \Bigg( \Bigg|\int_{0}^{1}\tilde{g}(s,t)\sqrt{w(s)}dW(s,\omega)-\int_{0}^{1} \tilde{t}_n(s,t)\sqrt{w(s)}dW(s,\omega)\Bigg|^{2}\Bigg)
\\
&=& E \Bigg( \Bigg|\int_{0}^{1}(\tilde{g}(s,t)-\tilde{t}_n(s,t))\sqrt{w(s)}dW(s,\omega)\Bigg|^{2}\Bigg)
\\
& = & \int_{0}^{1}\big|(\tilde{g}(s,t)-\tilde{t}_n(s,t))\big|^{2}w(s)ds,
\end{eqnarray*}
which converges to 0, since every $L^{2}_{w}(0, 1)$ function has a series expansion in transformed Hermite function.
Hence the theorem is proved.
\end{proof}
Let us denote the sum function $\int_{0}^{1}\tilde{g}(s, t)\sqrt{w(s)}dW(s, \omega)$ of the series $\sum_{n = 0}^{\infty}\tilde{b}_{n}\tilde{B}_{n}(\omega)\tilde{\psi}_{n}^{\alpha}(t)$ as $\tilde{F}(t, \omega)$.
\\Now the Fourier transform of $\tilde{F}(t, \omega)$ can be computed in terms of the Fourier transform of $\tilde{\psi}_{n}^{\alpha}(t)$ as
\begin{eqnarray}\label{2.22}
\mathcal{F}(\tilde{F}(t, \omega)) := \sum_{n = 0}^{\infty}\tilde{b}_{n}\tilde{B}_{n}(\omega)\mathcal{F}(\tilde{\psi}_{n}^{\alpha}(t)) = \sum_{n = 0}^{\infty}\tilde{b}_{n}\tilde{B}_{n}(\omega)\frac{1}{\sqrt{\alpha}}\lambda_{k}\tilde{\psi}_{n}(\frac{\phi(t)}{\alpha}),
\end{eqnarray}
where $\lambda_{k}$ are eigen value of the transformed Hermite function with absolute value one. The following theorem establishes the convergence of the series (\ref{2.22}).
\begin{theorem}
If $\tilde{g} \in L^{2}_{w}(0, 1)$ and $\tilde{b}_{n}$, $\tilde{B}_{n}$ are defined as in Theorem \ref{T2}
then the series
\begin{eqnarray}\label{2.15aa}
\sum_{n = 0}^{\infty}\tilde{b}_{n}\tilde{B}_{n}(\omega)\frac{1}{\sqrt{\alpha}}\lambda_{n}\psi_{n}(\frac{\phi(t)}{\alpha}),
\end{eqnarray}
converges in the sense of quadratic mean to the stochastic integral $\int_{0}^{1}\mathcal{F}(\tilde{g}(s, t))dW(t, \omega)$.
\end{theorem}
\begin{proof}
Denote the $n^{th}$ partial sum of the series (\ref{2.15aa}) as
\begin{eqnarray}\label{2.17a}
\tilde{S}_n(\tilde{t}, \omega):=
\sum_{k=0}^n \tilde{b}_{k}\tilde{B}_{k}(\omega)\frac{1}{\sqrt{\alpha}}\lambda_{k}\psi_{k}(\frac{\phi(t)}{\alpha}).
\end{eqnarray}
Its integral form is $\tilde{T}_{n}(t, \omega) = \int_{0}^{1}\big(\sum_{k = 0}^{\infty}\frac{1}{\sqrt{\alpha}}\tilde{b_{k}}\tilde{\psi}_{k}^{\alpha}(s)\lambda_{k}\psi_{k}(\frac{\phi(t)}{\alpha})\sqrt{w(s)}\big)dW(s, \omega)$.
Denote $\sum_{k = 0}^{\infty}\tilde{b}_{k}\tilde{\psi}_{k}^{\alpha}(s)\tilde{\psi}_{k}^{\alpha}(t)$ as $\tilde{g}(s, t)$.

The Fourier transform of $\tilde{g}(s, t)$ can be computed to be $\sum_{k = 0}^{\infty}\tilde{b}_{k}\tilde{\psi}_{k}^{\alpha}(s)\mathcal{F}(\tilde{\psi}_{k}^{\alpha}(t)) = \sum_{k = 0}^{\infty}\frac{1}{\sqrt{\alpha}}\tilde{b_{k}}\tilde{\psi}_{k}^{\alpha}(s)\lambda_{k}\psi_{k}(\frac{\phi(t)}{\alpha})$.
Let the partial sum of $\mathcal{F}(\tilde{g}(s, t))$ be $\sum_{k = 0}^{n}\frac{1}{\sqrt{\alpha}}\tilde{b_{k}}\tilde{\psi}_{k}^{\alpha}(s)\lambda_{k}\psi_{k}(\frac{\phi(t)}{\alpha})$ . Denote it as $\mathcal{F}(\tilde{t}_{n}(s, t)$, where $\tilde{t}_{n}(s, t) = \sum_{k = 1}^{n}\tilde{b}_{k}\tilde{\psi}_{k}^{\alpha}(s)\tilde{\psi}_{k}^{\alpha}(t)$.
Now $$\tilde{S}_{n}(t, \omega) = \int_{0}^{1}\mathcal{F}(\tilde{t}_{n}(s, t))\sqrt{w(s)}dW(s, \omega).$$
Since Fourier transform of a function in $L_{w}^{2}(0, 1)$ is in $L_{w}^{2}(0, 1)$, the stochastic integral $\int_{0}^{1}\mathcal{F}(\tilde{t}_{n}(s, t))\sqrt{w(s)}dW(s, \omega)$ exists in the sense of convergence in quadratic mean.
\\Now,
\begin{eqnarray*}
&& E \Bigg|\int_{0}^{1}\mathcal{F}(\tilde{g}(s, t))\sqrt{w(s)}dW(t, \omega) - \tilde{T}_{n}(t, \omega)\Bigg|^{2}
\\
& = &E \Bigg|\int_{0}^{1}\mathcal{F}(\tilde{g}(s,t))\sqrt{w(s)}dW(s,\omega)-\int_{0}^{1} (\mathcal{F}\tilde{t}_n(s,t))\sqrt{w(s)}dW(s,\omega)\Bigg|^{2}
\\
& = & \int_{0}^{1}\big|\mathcal{F}(\tilde{g}(s,t)) - \mathcal{F}\tilde{t}_n(s,t)\big|^{2}w(s)ds
\\
& = & \int_{0}^{1}\big|\tilde{g}(s,t) - \tilde{t}_n(s,t)\big|^{2}w(s)ds
\end{eqnarray*}
which converges to 0.
Hence the theorem is proved.
\end{proof}
Let us call (\ref{2.15aa}) to be the random Fourier - Hermite transform of $f \in L_{w}^{2}(0, 1)$ associated with the Wiener process.
\section{\textbf{Random Fourier - Hermite transform associated with symmetric stable process}}
\setcounter{equation}{0}
Let $X(t, \omega)$ be a symmetric stable process of index $\gamma$, $1 < \gamma \leq 2$ and let $f$ be a function in $L^{\gamma}(\mathbb{R})$. The existence of the stochastic integral $\int_{-\infty}^{\infty}f(t)dX(t, \omega)$ is established in Theorem \ref{12} below. Following Lemma is used to prove this.

\begin{lemma} \label{l1} \cite{15}
If $X(t,\omega)$ for $t \in \mathbb{R}$ is a symmetric stable process with independent increment of index $\gamma,$ for $1 < \gamma \leq 2$ and $f \in L^{p}[a, b]$, $p\geq \gamma,$ then the following inequality holds:
$$E\Bigg(\Bigg|\int_a^b f(t)dX(t,\omega)\Bigg| \Bigg)
\leq \frac{4}{\pi(\gamma-1)}
\int_a^b |f(t)|^\gamma dt
+\frac{2}{\pi}
\int_{|u|>1}\frac{1-exp(-|u|^\gamma \int_a^b|f(t)|^\gamma dt)}{u^2}
du.$$
\end{lemma}

\begin{theorem}\label{12}
If $X(t, \omega)$ for $t \in \mathbb{R}$ is a symmetric stable process of index $\gamma$, $1 < \gamma \leq 2$ and $f \in L^{1}(\mathbb{R})\cap L^{\gamma}(\mathbb{R})$, then the stochastic integral $\int_{-\infty}^{\infty}f(t)dX(t, \omega)$ exists in the sense of convergence in mean.
\end{theorem}
\begin{proof}
We know that $C_{c}(\mathbb{R})$ is dense in $L^{\gamma}(\mathbb{R})$, $1 \leq \gamma < \infty$.
For $f \in L^{1}(\mathbb{R})\cap L^{\gamma}(\mathbb{R})$, choose $f_{n}(t) = f(t)\chi_{[-n, n]}$ and $f_{m}(t) = f(t)\chi_{[-m, m]}$ in $C_{c}(\mathbb{R})$. Both $f_{n}$ and $f_{m}$ can be considered to be in $L^{\gamma}[-n, n]$ if $n > m$. Consider the two stochastic integrals $Y_{n}(\omega) = \int_{-\infty}^{\infty}f_{n}(t)dX(t, \omega)$ and $Y_{m}(\omega) = \int_{-\infty}^{\infty}f_{m}(t)dX(t, \omega)$, which exist in the sense of convergence in mean \cite{6}. \\Now using Lemma (\ref{l1}), we get
\begin{eqnarray*}
E \big|Y_{n} - Y_{m}\big| &=& E \Bigg| \int_{-\infty}^{\infty}f_{n}(t)dX(t, \omega) - \int_{-\infty}^{\infty}f_{m}(t)dX(t, \omega)\Bigg|
\\
&=& E \Big|\int_{-\infty}^{\infty}(f_{n}(t) - f_{m}(t))dX(t, \omega)\Big|
\\
&\leq & E \Bigg|\int_{-n}^{n}|f_{n}(t)-f_{m}(t)|dX(t, \omega) \Bigg|,~~\textnormal{since}~~f_{n}(t)-f_{m}(t)~~L^{\gamma}[-n, n]
\\
&\leq & \frac{4}{\pi} \int_{-n}^{n}|f_{n}(t)-f_{m}(t)|^{\gamma}dt + \frac{2}{\pi} \int_{|u| > 1} \frac{1 - \exp \big( -|u|^{\gamma}\int_{-n}^{n}|f_{n}(t)-f_{m}(t)|^{\gamma}dt\big)}{u^{2}}du .
\end{eqnarray*}
The integral $\int_{-n}^{n}|f_{n} - f_{m}|^{\gamma}dt = \int_{-n}^{-m}|f|^{\gamma}dt + \int_{-m}^{n}|f|^{\gamma}dt $ approaches $0$ as $m, n \rightarrow \infty$.
The integrand in the $2^{nd}$ integral is dominated by the integrable function $\frac{1}{u^{2}}$ over $(-\infty, 1]$ and $[1, \infty)$. Hence by dominated convergence theorem $$ \lim_{m, n \rightarrow \infty} E |Y_{n} - Y_{m}| = 0.$$ Now $Y_{n}$ is a Cauchy sequence which will converge to $Y$ (say) and denote $Y = \int_{-\infty}^{\infty}f(t)dX(t, \omega)$.
The existence of the stochastic integral $\int_{-\infty}^{\infty}f(t)dX(t,\omega)$ in the sense of mean is thus established.
\end{proof}
 If $f(t)$ is the Hermite function $\psi_{n}(t)$, then the integral $\int_{-\infty}^{\infty}\psi_{n}(t)dX(t, \omega)$ will exist, and is a random variable. Denote $A_{n}(\omega) = \int_{-\infty}^{\infty}\psi_{n}(t)dX(t, \omega)$. These $A_{n}(\omega)$ are not independent. It is proved in Lemma\ref{Le7} below by establishing the fact that the characteristic function of $(A_{n} + A_{m})$ is not equal to the product of characteristic function of $A_{n}$ and the characteristic function $A_{m}$. The characteristic function of $A_{n}$ is computed in the following Theorem.
\begin{theorem}\label{t6}
The characteristic function of $\int_{-\infty}^{\infty}f(t)dX(t, \omega)$ is $\exp \big(-|u|^{\gamma}\int_{-\infty}^{\infty}|f(t)|^{\gamma}dt\big)$ where $f(t) \in L^{\gamma}(\mathbb{R})$, $1 < \gamma \leq 2$.
\end{theorem}
\begin{proof}
We know that every $f \in L^{\gamma}(\mathbb{R})$, $\gamma \geq 1$ can be approximated by $f_{k} = f\chi_{[-k, k]}$ in $C_{c}(\mathbb{R})$. Hence ${f_{k}}$ form a cauchy sequence in $C_{c}(\mathbb{R})$.
We know that the characteristic function of the stochastic integral $\int_{-\infty}^{\infty}f_{k}(t)dX(t, \omega)$ is $C_{k} = e^{\big(-c|u|^{\gamma}\int_{-k}^{k}|f(t)|^{\gamma}dt\big)}  $ \cite{11}.
\\Now for $k > l$
\begin{eqnarray*}
C_{k} - C_{l} & = & \exp \big(-c|u|^{\gamma}\int_{-\infty}^{\infty}|f_{k}(t)|^{\gamma}dt \big) - \exp \big(-c|u|^{\gamma}\int_{-\infty}^{\infty}|f_{l}(t)|^{\gamma}dt\big)
\\& = & \exp \big(-c|u|^{\gamma}\int_{-\infty}^{\infty}|f_{l}(t)|^{\gamma}dt\big) \Big\{\exp \big(-c|u|^{\gamma}\int_{-\infty}^{\infty}(|f_{k}(t)|^{\gamma} - |f_{l}(t)|^{\gamma})dt \big) - 1 \Big\}
\\&\leq & \exp \big(-c|u|^{\gamma}\int_{-\infty}^{\infty}|f_{l}(t)|^{\gamma}dt\big) \Big\{\exp \big(-c|u|^{\gamma}\int_{-k}^{k}(|f_{k}(t)|^{\gamma} - |f_{l}(t)|^{\gamma})dt \big) - 1 \Big\}
\end{eqnarray*}
Since $$\int_{-k}^{k}(|f_{k}(t)|^{\gamma} - |f_{l}(t)|^{\gamma})dt \rightarrow  0.$$

as $k,l \rightarrow \infty$, we have $C_{k} - C_{l}$ converges to 0.
Hence ${C_{k}}$ is a cauchy sequence in $L^{\gamma}(\mathbb{R})$. Now since $f_{k}$ converges to $f$, it can be shown that $C_{k}$ converges to $c|u|^{\gamma}\int_{-\infty}^{\infty}|f(t)|^{\gamma}dt \big)$. Define $\exp \big(-c|u|^{\gamma}\int_{-\infty}^{\infty}|f(t)|^{\gamma}dt \big)$ to be the characteristic function of $\int_{-\infty}^{\infty}f(t)dX(t, \omega)$.
\end{proof}

\begin{theorem}\label{Le7}
If $X(t, \omega)$, $t \in \mathbb{R}$ is a symmetric stable process of index $\gamma$, $1 < \gamma \leq 2$, then the random variables $A_{n} = \int_{-\infty}^{\infty}\psi_{n}(t)dX(t, \omega)$ associated with $X(t, \omega)$ are not independent.
\end{theorem}
\begin{proof}
 By Theorem \ref{t6},
the characteristic function of $A_{n}(\omega)$ is $\exp \big(-c|u|^{\gamma}\int_{-\infty}^{\infty}|\psi_{n}(t)|^{\gamma}dt \big)$.
Hence, the characteristic function of $(A_{n}(\omega) + A_{m}(\omega))$ is $\exp \big( -c|u|^{\gamma}\int_{-\infty}^{\infty}|\psi_{n}(t) + \psi_{m}(t|^{\gamma}dt \big)$, where as the product of characteristic function of $A_{n}(\omega)$ and the characteristic function of $A_{m}(\omega)$ is
$\exp \big(-c|u|^{\gamma}\int_{-\infty}^{\infty}|\psi_{n}(t)|^{\gamma}dt \big)\exp \big(-c|u|^{\gamma}\int_{-\infty}^{\infty}|\psi_{m}(t)|^{\gamma}dt \big) = \exp \big(-c|u|^{\gamma}\int_{-\infty}^{\infty}(|A_{n}(t)|^{\gamma} + |A_{m}(t)|^{\gamma} ) dt \big)$
which are clearly not equal. Hence, $A_{n}(\omega)$ are not independent random variables.
\end{proof}

Consider the random Hermite series
\begin{eqnarray}\label{1aa}
\sum_{n=0}^\infty a_{n}A_n(\omega)\psi_{n}(t),
\end{eqnarray}
where $\psi_{n}$ are normalized  Hermite-Gaussian functions, the scalars $a_{n}$ are the Fourier - Hermite coefficients of a function $f$ and $A_n(\omega)$ are random variables defined as
\begin{eqnarray}\label{1.15a}
A_{n}(\omega) := \int_{-\infty}^{\infty}\psi_{n}(t)dX(t, \omega),
\end{eqnarray}
which are the Fourier - Hermite coefficients of symmetric stable process $X(t, \omega)$ of index $\gamma$, $1 < \gamma \leq 2$.

The inequality in the following Lemma is required to prove the convergence of RFH series (\ref{1aa}).
\begin{lemma}\label{13}
Let $f$ be any function in $L^{1}(\mathbb{R}) \cap L^{\gamma}(\mathbb{R})$
and $X(t,\omega)$, $t \in \mathbb{R}$ be a symmetric stable process of index $\gamma$, $1 < \gamma \leq 2$, then
$$E \Bigg( \Bigg|\int_{-\infty}^{\infty} f(t)dX(t,\omega)\Bigg| \Bigg)
\leq \frac{4}{\pi(\gamma - 1)}
\int_{-\infty}^{\infty}|f(t)|^{\gamma} dt + \frac{2}{\pi} \int_{|u| > 1} \frac{1 - \exp \big( -|u|^{\gamma}\int_{-\infty}^{\infty}|f(t)|^{\gamma}dt\big)}{u^{2}}du .
$$
\end{lemma}
To prove it, we require the following two results.
\begin{lemma}\cite{1e}\label{13a}
A stable random variable $X(t, \omega)$ always satisfies the inequality $E|X|^{r} < \infty$ for all $r \in (0, \gamma)$, $0 < \gamma \leq 2$.
\end{lemma}
\begin{lemma}\cite{1d}\label{13b}
If $\phi$ is the characteristic function of a random variable $X$, then $E|X| = \int_{}^{}|X|dF(X) = \frac{2}{\pi}\int_{-\infty}^{\infty}\frac{1 - Re \phi(t)}{t^{2}}dt$, where $F(X)$ is the distribution function of $X$.
\end{lemma}

\begin{proof}
\textbf{(Proof of Lemma \ref{13})}
\\We know that $\int_{-\infty}^{\infty}f(t)dX(t, \omega)$ exists in the sense of mean (Theorem \ref{12}). Now using Lemma \ref{13a} and \ref{13b}, we have
\begin{eqnarray*}
E \Bigg( \Bigg|\int_{-\infty}^{\infty} f(t)dX(t,\omega)\Bigg|\Bigg) &=& \frac{2}{\pi}\int_{-\infty}^{\infty}\frac{1 - Re \psi(u)}{u^{2}}du
\\& = & \frac{2}{\pi}\int_{|u| \leq 1}\frac{1 - Re \psi(u)}{u^{2}}du + \frac{2}{\pi}\int_{|u| > 1}\frac{1 - Re \psi(u)}{u^{2}}du.
\end{eqnarray*}
Here
\begin{eqnarray*}
\int_{|u| \leq 1}\frac{1 - Re \psi(u)}{u^{2}}du &=& \int_{-1}^{1} \frac{1 - \exp \big( -|u|^{\gamma}\int_{-\infty}^{\infty}|f(t)|^{\gamma}dt\big)}{u^{2}}du
\\&\leq& \int_{-1}^{1} \frac{|u|^{\gamma}\int_{-\infty}^{\infty}|f(t)|^{\gamma}dt}{u^{2}}du ~~(~~\because~~ 1~~ - ~~e^{-x} ~~<~~ x~~ \textnormal{for}~~ x~~>~~0)
\\&=& 2\int_{0}^{1} |u|^{\gamma - 2}du\int_{-\infty}^{\infty}|f(t)|^{\gamma}dt
\\&=& \frac{2}{\gamma - 1}\int_{-\infty}^{\infty}|f(t)|^{\gamma}dt
\end{eqnarray*}
Hence we have
\begin{eqnarray*}
E \Bigg( \Bigg|\int_{-\infty}^{\infty} f(t)dX(t,\omega)\Bigg|\Bigg) \leq   \frac{4}{\pi(\gamma - 1)}
\int_{-\infty}^{\infty}|f(t)|^{\gamma} dt + \frac{2}{\pi} \int_{|u| > 1} \frac{1 - \exp \big( -|u|^{\gamma}\int_{-\infty}^{\infty}|f(t)|^{\gamma}dt\big)}{u^{2}}du.
\end{eqnarray*}
\end{proof}

The following Theorem establishes the convergence of the series (\ref{1aa}),
to the stochastic integral
\begin{eqnarray}\label{1b}
\int_{-\infty}^{\infty}f(s, t)dX(s, \omega),
\end{eqnarray}
in the sense of mean, if $a_{n}$ are the Fourier-Hermite coefficients of $f(s) \in L^{1}(\mathbb{R}) \cap L^{\gamma}(\mathbb{R})$, $\frac{4}{3}< \gamma \leq 2$.

\begin{theorem}
If $X ( s, \omega), s \in \mathbb{R}$ is a symmetric stable process of index $\gamma$, $\frac{4}{3} <\gamma \leq 2$, then the series (\ref{1aa}) with random coefficients $A_{n}(\omega)$ defined as in (\ref{1.15a}), converges in the mean to the stochastic integral (\ref{1b}) if $a_{n}$  are the Fourier - Hermite coefficients of a function $f(t) \in L^{1}(\mathbb{R}) \cap L^{\gamma}(\mathbb{R})$, defined as $a_{n} = \int_{-\infty}^{\infty}f(t)\psi_{n}(t)dt$.
\end{theorem}
\begin{proof}
 We know that, every $f \in L^{\gamma}(\mathbb{R})$, $\frac{4}{3} < \gamma < 4$ has the Fourier-Hermite series expansion $\sum_{n = -\infty}^{\infty}a_{n}\psi_{n}(t)$ \cite{1a}. Let the partial sum of this Fourier-Hermite series be $s_n(t) := \sum_{k=0}^n a_k \psi_k(t)$.\\Denote the partial sum of the series $\sum_{k=0}^\infty a_k \psi_k(s)\psi_k(t)$ as $s_n(s,t) := \sum_{k=0}^n a_k \psi_k(s)\psi_k(t).$
Let,
$$S_n(t,\omega) :=
\sum_{k=0}^n a_k
A_k(\omega)\psi_k(t)$$
be the $n^{th}$ partial sum of the series (\ref{1aa}) associated with symmetric stable process $X(t, \omega)$ and
\\$f \in L^{1}(\mathbb{R}) \cap L^{\gamma}(\mathbb{R})$, $\frac{4}{3} < \gamma \leq 2$.
\\In integral form,
\begin{eqnarray*}
S_n(s,\omega) &=&
\sum_{k=0}^n a_k
A_k(\omega)\psi_k(t)
\\
&=& \sum_{k=0}^n a_k
\Bigg(\int_{-\infty}^{\infty}\psi_k(s)dX(s, \omega)\Bigg) \psi_k(t)
\\
&=& \int_{-\infty}^{\infty}\Big(\sum_{k=0}^n a_k \psi_n(s)\psi_n(t)\Big)dX(s, \omega)
\\
&=& \int_{-\infty}^{\infty}s_n(s,t)dX(s,\omega).
\end{eqnarray*}
\\We know that $\int_{-\infty}^{\infty}f(s, t) dX(s, \omega)$ exist in the sense of convergence in mean( Theorem \ref{12}).
\\Now,
\begin{eqnarray*}
&&E \Bigg( \Bigg|\int_{-\infty}^{\infty} f(s,t)dX(s,\omega)-S_n(s,\omega) \Bigg| \Bigg)
\\&=& E \Bigg( \Bigg|\int_{-\infty}^{\infty}f(s,t)dX(s,\omega)-\int_{-\infty}^{\infty} s_n(s,t)dX(s,\omega)\Bigg|\Bigg)
\\
&=& E \Bigg( \Bigg|\int_{-\infty}^{\infty}(f(s,t)-s_n(s,t)dX(t, \omega))
\\
&\leq & \frac{4}{\pi(\gamma - 1)}
\int_{-\infty}^{\infty} |(f(s,t)-s_n(s,t)|^{\gamma} ds \\&+& \frac{2}{\pi} \int_{|u| > 1} \frac{1 - \exp \big( -|u|^{\gamma}\int_{-\infty}^{\infty}|(f(s,t)-s_n(s,t)|^{\gamma}ds\big)}{u^{2}}du~~(~~\textnormal{by~~Lemma~~} \ref{13}) .
\end{eqnarray*}
   Since the integral $\int_{-\infty}^{\infty}|f(s, t) - s_{n}(s, t)|^{\gamma}ds$ tends to 0 as $n \rightarrow \infty$ if $\gamma$ lies in the interval $(1, 2]$ and the integrand in the second integral is dominated by the integrable function $\frac{1}{u^{2}}$, the theorem is proved by applying (\ref{1.8}).
\end{proof}
Denote the sum function  $\int_{-\infty}^{\infty}f(s, t)dX(t, \omega)$ of the series (\ref{1aa}) as $F(t, \omega)$. The series (\ref{1aa}) is the random Fourier - Hermite series(RFHS), associated with the symmetric stable process, $X(t, \omega)$.
\begin{theorem}\label{t2}
If $X ( s, \omega), s \in \mathbb{R}$ is a symmetric stable process of index 2, $\mathcal{F}$ is the Fourier transform operator, then for all function $f(t) \in L^{1}(\mathbb{R}) \cap L^{2}(\mathbb{R})$, the series
\begin{eqnarray}\label{1d}
\sum_{0}^{\infty}a_{k}\lambda_{k}A_{k}(\omega)\psi_{k}(t),
\end{eqnarray}
converges in the mean to the stochastic integral
\begin{equation}\label{22}
\int_{-\infty}^{\infty} \mathcal{F}(f(s,t))dX(s,\omega),
\end{equation}
where $\lambda_{k} = e^{-\frac{ik\pi}{2}}$, $k \in \mathbb{N}_{0}$ are the eigen values of the Fourier transform   $\mathcal{F}(f)$, where $a_{k}$ and $A_{k}(\omega)$ are defined as in (\ref{1.8a}) and (\ref{1.15a}) respectively.
\end{theorem}
\begin{proof}
Let
$$S_n(s,\omega) :=
\sum_{k=0}^n a_k
\lambda_{k}A_k(\omega)\psi_k(t)$$
be the $n^{th}$ partial sum of the series (\ref{1d}).\\ Denote the partial sum of the series $\sum_{k = 0}^{\infty}a_{k}\lambda_{k}\psi_{k}(t)$ as
$\mathcal{F}(s_{n}(t)) := \sum_{k=0}^n a_k \lambda_{k} \psi_k(t)$ which is the Fourier transform of $s_{n}(t)$ and let
\begin{eqnarray*}
\mathcal{F}(s_n(s,t)) := \sum_{k=0}^n a_k \lambda_{k}\psi_k(s)\psi_k(t).
\end{eqnarray*}
In integral form,
\begin{eqnarray*}
S_n(t, \omega) &=&
\sum_{k=0}^n a_k
\lambda_{k}A_k(\omega)\psi_k(t)
\\
&=& \sum_{k=0}^n a_k \lambda_{k}
\Bigg(\int_{-\infty}^{\infty}\psi_k(s)dX(s, \omega)\Bigg) \psi_k(t)
\\
&=& \int_{-\infty}^{\infty}\mathcal{F}(s_n(s,t))dX(s,\omega).
\end{eqnarray*}
\\Since the integral $\int_{-\infty}^{\infty}\big|\mathcal{F}(f(s, t)) - \mathcal{F}(s_{n}(s, t))\big|^{2}ds$ converges to 0 for $f \in L^{1}(\mathbb{R}) \cap L^{2}(\mathbb{R})$, it can be shown that
$E\Bigg(\Big|\int_{-\infty}^{\infty}\mathcal{F}[f(s,t)]dX(s,\omega)-S_n(s,\omega) \Big| \Bigg)$
 converges to 0 as $n \rightarrow \infty$ by using the same argument as in the previous theorem. Hence, it is proved that the random series (\ref{1d}) converges in mean to $\int_{-\infty}^{\infty}\mathcal{F}[f(s,t)]dX(s,\omega)$.
\end{proof}
We call the series (\ref{1d}) as the random Fourier - Hermite transform of $L^{2}(\mathbb{R})$ associated with the symmetric stable process, $X(t, \omega)$.
\section*{Acknowledgments}
This research work was supported by University Grant Commission (Rajiv Gandhi National Fellowship with letter no-F./2015-16/RGNF-SC-2015-16-SC-ORI-20053).

\end{document}